\documentclass[12pt,a4paper]{amsart}
\usepackage{amsfonts,amsmath,amsthm,amssymb}
\usepackage{pstricks,pst-node}
\usepackage{fullpage}
\usepackage{dsfont}

\newcommand\N{\mathds{N}}
\newcommand\R{\mathds{R}}

\newcommand\mcirc{\!\bigcirc}

\newcommand\wc{\red\mcirc\!}
\newcommand\wcp{\red\mcirc}
\newcommand\bs{\raisebox{-1.25pt}{\blue$\blacksquare$}}
\newcommand\gr[1]{\rlap{\,\gray#1}\hphantom{\wc}}
\newcommand\grr[1]{\rlap{\,\!\!\gray#1}\hphantom{\wc}}
\newcommand\ttilde{\raise.17ex\hbox{$\scriptstyle\mathtt{\sim}$}}

\newtheorem{theorem}{Theorem}[section]
\newtheorem{definition}[theorem]{Definition}
\newtheorem{lemma}[theorem]{Lemma}
\newtheorem{corollary}[theorem]{Corollary}

\newtheorem{example}[theorem]{Example}

\newcommand\nell{k}

\title{New developments of an old identity}

\author{Rui Duarte}
\address{Center for Research and Development in Mathematics and Applications,
Department of Mathematics,
University of Aveiro}
\email{rduarte@ua.pt}

\author{Ant\'onio Guedes de Oliveira} \address{CMUP and Mathematics
  Department, Faculty of Sciences, University of Porto}
\email{agoliv@fc.up.pt}

\date{\today}

\begin{document}

\begin{abstract}
We give a direct combinatorial proof of a famous identity,
\begin{equation}\label{mainId}
  \sum_{i+j=n}
  \binom{2i}{i}\binom{2j}{j}\;=\;\displaystyle4^n \end{equation}
by actually counting pairs of $\ell$-subsets of $2\ell$-sets. Then we
discuss two different generalizations of the identity, and end the paper
by presenting in explicit form the ordinary generating function of the
sequence $\left(\strut\binom{2n+\ell}{n}\right)_{n\in\N_0}$, where $\ell\in\R$.
\end{abstract}

\maketitle

\section{Introduction and definitions}
\noindent
In \cite[Exercise 1.4(a)]{RS}, Richard Stanley asks for a proof of the
identity:
$$(1-4x)^{-1/2}= \sum_{n\geq 0}\binom{2n}{n}\,x^n$$
from which \eqref{mainId} above follows immediately for all
nonnegative integers $i$, $j$ and $n$ \cite[Exercise 1.2(c)]{RS}.  In
\cite[Exercise 1.4(b)]{RS}, it is asked for the value of $\sum_{n\geq
  0}\binom{2n-1}{n}\,x^n$. The answer is one half plus one half of the
previous sum, and thus
\begin{equation}\label{eq1}
 \sum_{\substack{i\geq1\\i+j=n}}\binom{2i-1}{i}\binom{2j-1}{j}\;=\;\displaystyle4^{n-1}
  \quad(n>1).
\end{equation}
By replacing $i$ with $i+1$ and $n$ with $n+1$ in \eqref{eq1}, we may
see that this identity, as well as \eqref{mainId}, is a special case
of the new identity that we will prove --- in the notation of
\cite{RS}, where, in particular, for every real $\ell$ and
every nonnegative integer $i$, $\big(\ell\big)_i=\ell(\ell-1)\dotsb
(\ell-i+1)$ (the falling factorial) and
$\binom{\ell}{i}=\frac{\strut(\ell)_i}{\strut i!}$,
\begin{equation}\label{eq2}
  \sum_{i+j=n}\binom{2i-\ell}{i}\binom{2j+\ell}{j}\;=\;\displaystyle4^n.
\end{equation}
Our proof of the new identity does not use \eqref{mainId}.  On the
contrary, it proves the initial identity by using the
inclusion-exclusion principle for positive values of $\ell$.  Thus, in
particular, we give a new solution to the problem of finding a
combinatorial proof of \eqref{mainId}, that was posed and solved a
long time ago, since, according to Paul Erd\H{o}s, ``Hungarian
mathematicians tackled it in the thirties: P.  Veress proposing and
G. Hajos solving it'' \cite{MS}.  This is told by Marta Sved, who,
after having posed the problem of providing a combinatorial proof of
the identity in a previous article \cite{MSa}, describes a number of
answers received meanwhile as follows: ``All solutions are based, with
some variations, on the count of lattice paths, or equivalently
$(1,0)$ sequences''. Although this is neither the case of the cited
articles \cite{VdA,ChX}, for example, nor the case of the short proof
of Lemma~ \ref{newnewlemma}, below, a combinatorial proof is still
missing where $\binom{2i}{i}$ really stands for the number of
$i$-subsets of a $(2i)$-set.

We give here this proof, based on the following idea: suppose we
associate to a given summand on the left-hand side of \eqref{mainId}
the pair $(A,B)$, where $A\subseteq[2i]$, $|A|=i$, $B\subseteq[2j]$
and $|B|=j$, being $[\ell]=\{1,2,\dotsc,\ell\}$ for a nonnegative integer
$\ell$, as usual. We view the two intervals $[2i]$ and $[2j]$ as the sets
of slots of two rectangles, of size $2\times i$ and $2\times j$,
respectively, ordered row by row and joined in a $(2\times
n)$-rectangle $R$, the first on the left-hand side and the second on
the right-hand side. We color the elements of $A$ with one color and
the elements of $B$ with another color.  Hence, for example, we
represent $\big(\{2,3,4,8,9\},\{2,3,4,7,8,10\}\big)$ by
$$R\;=\;
\begin{array}{|c|c|c|c|c||c|c|c|c|c|c|}
  \hline
  \gr{1}&\wc&\wc&\wc&\gr{5}&\gr{1}&\bs&\bs&\bs&\gr{5}&\gr{6}\\
  \hline
  \gr{6}&\gr{7}&\wc&\wc&\grr{10}&\bs&\bs&\gr{9}&\bs&\grr{11}&\grr{12}\\
  \hline
\end{array}\;.$$

\begin{definition}
  \emph{We call \emph{configuration} (or \emph{$n$-configuration}) to
    a $(2\times n)$-rectangle where exactly $n$ of the $2n$ slots are
    painted with one of two colors, with the restriction that columns
    with both colors do not exist. The pair $(i,j)$ for a
    $(i+j)$-configuration $R$ with $i$ slots colored with color one
    and $j$ slots colored with color two (or $i$ slots filled with
    $\wcp$ and $j$ slots filled with $\bs$) is the \emph{type} of
    $R$.}  \emph{The columns with no colored slots are \emph{empty
      columns} and the columns with two colored slots
    \emph{towers}. Both are called \emph{even columns} and the columns
    with exactly one colored slot are called \emph{odd}.}  \emph{A
    type $(i,j)$ configuration is \emph{ordered} if the $i$ slots with
    color one belong to the leftmost $(2\times i)$-subrectangle and
    the $j$ slots of color two belong to the complementary $(2\times
    j)$-subrectangle. Finally, the set of ordered $n$-configurations
    is denoted $\mathcal{O}_n$ and the set of $n$-configurations
    without towers is denoted $\mathcal{T}_n$.}
\end{definition}

Note that the ordered configurations are exactly the
configurations that represent the pairs $(A,B)$ as defined above.
By definition, the number of towers and the number of empty columns in
each of the two original subrectangles are equal. In the example
above, the type is $(5,6)$, the empty columns have \emph{order} $1$,
$5$, $10$ and $11$ and the towers $3$, $4$, $7$ and $9$.

Note that $4^n$ is the total number of $n$-configurations
\emph{without} towers.  Hence, what \eqref{mainId} says is that the
number of ordered configurations equals the number of configurations
without towers, and in fact we define (recursively) a bijection
$\varphi$ between these two sets that leaves the ordered
configurations without towers invariant. More precisely, if $R$ is an
ordered configuration with $\nell$ towers then $\varphi(R)$ is a
configuration where exactly in $\nell$ cases a column of color two is
followed by a column of color one. The algorithm beneath the proof was
implemented in \emph{Mathematica} and can be obtained in
\cite{algoritmo}.  For example, for the configuration $R$ above,
$$\varphi(R)\;=\;
\begin{array}{|c|c|c|c|c|c|c|c|c|c|c|}
  \hline
  \bs&\gr{}&\wc&\gr{}&\gr{}&\gr{}&\bs&\gr{}&\wc&\gr{}&\wc\\
  \hline
  \gr{}&\bs&\gr{}&\bs&\wc&\bs&\gr{}&\wc&\gr{}&\bs&\gr{}\\
  \hline
\end{array}\;.$$

\medskip

After this new proof, in Section 3 we prove \eqref{eq2} and another
generalization of \eqref{mainId}: given integers $n\geq0$ and $t>0$,
define
$$
S_t(n) = \sum_{i_1 + \cdots + i_t = n} \binom{2i_1}{i_1} \binom{2i_2}{i_2} \dotsb \binom{2i_t}{i_t},
$$
where $i_1,\dotsc,i_n$ are any nonnegative integers (note that
\eqref{mainId} states that $S_2(n)=4^n$).  Only very recently
\cite{ChX}, Guisong Chang and Chen Xu showed, with a probabilistic
proof, that $S_t(n)$ depends only on $n$ and $t$. We give here a
combinatorial proof of this fact and obtain a generalization that also
includes \eqref{eq2} as a special case.  Finally, in Section 4 we
obtain explicitly the generating functions of the sequences involved
in these identities.

\section{New proof of the main identity}
\noindent
Given an $n$-configurations $R$ with $\nell$ towers, let $R'$ be
$(2\nell)$-configuration obtained from $R$ by removing all the odd
columns.  For a $(2\nell)$-configuration $S$ without odd columns, if
we delete one of the two equal rows of $S$ and then rearrange the
remaining $2\nell$ slots by placing, for every $1\leq i\leq\nell$,
slot $2i$ under slot $2i-1$, we obtain a $\nell$-rectangle
$S_\downarrow$ called the \emph{compression of $S$}. In the opposite
direction, given an $\nell$-configuration $T$, by doubling each slot
of $T$ column by column, downwards and from left to right, we obtain
the \emph{expansion} of $T$, denoted $T^\uparrow$. The
\emph{tower-configuration of $R$} is $R'_\downarrow$.

Note that all these three operations, when applied to an ordered
configuration, still give an ordered configuration.  For example, for
the $11$-configuration $R$ above,
$$R'=\begin{array}{|c|c|c|c||c|c|c|c|}
  \hline
  \gr{1}&\wc&\wc&\gr{4}&\bs&\bs&\gr{3}&\gr{4}\\
  \hline
  \gr{}&\wc&\wc&\gr{}&\bs&\bs&\gr{}&\gr{}\\
  \hline
\end{array}\,,\
R'_\downarrow=\begin{array}{|c|c||c|c|}
  \hline
  \gr{1}&\wc&\bs&\gr{3}\\
  \hline
  \wc&\gr{4}&\bs&\gr{4}\\
  \hline
\end{array}
\ \text{and}\
\left(R'_\downarrow\right)'_\downarrow=\begin{array}{||c|}
\hline\bs\\\hline\gr{}\\\hline
\end{array}\;.$$

In the proof of Theorem~\ref{theo}, we will need the following
definition:
\begin{definition}
\label{defpeq}
  \emph{ Let $R$ be a $n$-configuration where all but the first and
    the last columns are odd, and where either:
\begin{enumerate}
\item[(a)] \label{ala} The first column of $R$ is a tower of color $c$ and
  the last one is empty.
\item[(b)] \label{alb} The first column of $R$ is empty and the last one is
  a tower of color $c$.
\end{enumerate}
Then define $\phi_1 (R)$ and $\phi_2 (R)$ as follows:
\begin{enumerate}
\item  Color with the second color ($\bs$) one slot of the first column
  of $\phi_i(R)$, as follows: color the bottom one if $c=1$ and
  color the top one if $c=2$.
\item Color with the first color ($\:\wcp$) one slot of the $n$-th and
  last column of $\phi_i(R)$, as follows: color the bottom one in
  case~\eqref{ala} and color the top one in case~\eqref{alb}.
\item Color with the second color one slot of every column in-between
  of $\phi_1(R)$, the same as in the corresponding column of $R$;
  if the $(n-1)$-th column of $\phi_1(R)$ becomes different from
  the first one, change the position of the colored slot of every
  column but the last one.
\item Color with the first color one slot of every column in-between
  of $\phi_2(R)$, the same as in the corresponding column of $R$;
  if the second column of $\phi_2(R)$ becomes different from the
  last one, change the position of the colored slot of every column
  but the first one.
\end{enumerate}}
\end{definition}

We may now present our new proof of the theorem:
\begin{theorem}\label{theo}
  For every natural number $n$ there is a bijection
  $\varphi=\varphi_n:\mathcal{O}_n\to\mathcal{T}_n$.
\end{theorem}
\begin{proof}(By induction on $n$)\\[5pt]
  \textit{Definition of $\varphi$.}  Let $R\in\mathcal{O}_n$. If $R$
  has no towers, define $\varphi(R)=R$.  If $R$ has $\nell>0$ towers
  consider $S=R'$ and note that $S_\downarrow$, by induction, is in
  bijection with a configuration without towers, $T=\varphi_\nell(S)$,
  say; now, take the $(2\nell)$-configuration $U=T^\uparrow$ and
  replace every (even) column of $R'$ in $R$ by the corresponding
  (even) column of $U$, obtaining thus $R^*$, say. Since the tower
  configuration of $R^*$ is still $T$, which has no towers, the towers
  and empty columns of $R^*$ appear in pairs, meaning that for every
  $1\leq j\leq\nell$ either the $(2j-1)$-th even column is a tower and
  the $(2j)$-th even column is empty or vice versa. Note that, by
  construction, the orders of the $(2j-1)$-th and the $(2j)$-th
  columns are (a) both smaller or equal to the type $t$ of $R$ or they
  are (b) both greater than this number.

  Now, for every $1\leq j\leq\nell$, in the cases where (a) holds
  replace the section $R_j$ between the $(2j-1)$-th even column and
  the $(2j)$-th even column, included, by $\phi_1(R_j)$ as defined in
  Definition~\eqref{defpeq}, and in the cases where (b) holds replace
  it by  $\phi_2(R_j)$.\\[5pt]
  \textit{Bijectivity of $\varphi$.} By considering only the
  consecutive pairs formed by a column of color two followed by a
  column of color one of $\varphi(R)$, we obtain a
  $(2\nell)$-configuration without towers that encodes, according to
  Definition~\eqref{defpeq}, a unique $(2\nell)$-configuration $S$
  without odd columns.  Consider $T=S_\downarrow$ and take, by
  induction, $U=\big(\varphi_\nell^{-1}(T)\big)^\uparrow$.  Clearly,
  the towers of $U$ are the towers of $R$ and the column of color one
  of $\varphi(R)$ corresponding to the last even column of the same
  color of $R$ has the same order in $\varphi(R)$ as the latter in
  $R$, and the same happens with the column of color two of
  $\varphi(R)$ corresponding to the first even column of the same
  color in $R$. Hence, by induction we know the type of $R$ and we may
  use Definition~\ref{defpeq} to fully recover $R$.\end{proof}

\begin{example}\emph{ Since by Definition~\ref{defpeq}
$$\begin{array}{|c|c|c|c|}
\hline
\bs&\wc&\gr{}&\wc\\
\hline
\gr{}&\gr{}&\bs&\gr{}\\
\hline
\end{array}
\ \text{encodes}\
\begin{array}{|c|c|c|c|}
\hline
\gr{}&\bs&\gr{}&\wc\\
\hline
\gr{}&\bs&\gr{}&\wc\\
\hline
\end{array}
\quad\text{and}\quad
\begin{array}{|c|c|}
\hline\gr{}&\gr{}\\\hline\bs&\wcp\\\hline
\end{array}\ \text{encodes}\
\begin{array}{|c|c|}
\hline\wc&\gr{}\\\hline\wc&\gr{}\\\hline
\end{array}\,,$$
which contains no $\bs$-towers followed by $\wcp$-columns,
\begin{align*}
&\varphi^{-1}\left(\;
\begin{array}{|c|c|}
\hline\gr{}&\gr{}\\\hline\bs&\wcp\\\hline
\end{array}\;\right)=
\begin{array}{|c|c|}
\hline\wc&\gr{}\\\hline\wc&\gr{}\\\hline
\end{array}\;,\\[5pt]
&\varphi^{-1}\left(\;
\begin{array}{|c|c|c|c|}
\hline
\bs&\wc&\gr{}&\wc\\
\hline
\gr{}&\gr{}&\bs&\gr{}\\
\hline
\end{array}\;\right)=
\begin{array}{|c|c|c|c|}
\hline
\wc&\wc&\gr{}&\gr{}\\
\hline
\wc&\wc&\gr{}&\gr{}\\
\hline
\end{array}
\quad\text{and}\\[5pt]
&\varphi^{-1}\bigg(\;\begin{array}{|c|c|c|c|c|c|c|c|c|c|c|c|c|}
\hline
\gr{}&\bs&\bs&\wcp&\gr{}&\gr{}&\bs&\gr{}&\wc&\bs&\bs&\bs&\gr{}\\
\hline
\wc&\gr{}&\gr{}&\gr{}&\wc&\wc&\gr{}&\bs&\gr{}&\gr{}&\gr{}&\gr{}&\bs\\
\hline
\end{array}\;\bigg)=\\[2.5pt]
&\hphantom{\varphi^{-1}\bigg(}\;\begin{array}{|c|c|c|c|c|c|c|c|c||c|c|c|c|}
\hline
\gr{}&\wcp&\wcp&\wcp&\gr{}&\gr{}&\gr{}&\wcp&\gr{}&\bs&\bs&\bs&\gr{}\\
\hline
\wc&\wc&\gr{}&\wc&\wc&\wc&\gr{}&\gr{}&\gr{}&\gr{}&\gr{}&\gr{}&\bs\\
\hline
\end{array}
\end{align*} }
\end{example}

\section{On Chang-Xu generalization of the main identity}
\noindent
The results of this section can be shown using the generating
functions obtained in the following section. In the spirit of this
article, however, we present combinatorial proofs.

The main result of \cite{ChX} is a generalization of \eqref{mainId}
that we may write as:
\begin{equation} \label{CXa} \sum_{i_1 + \cdots + i_t
    =n}\binom{2i_1}{i_1}\binom{2i_2}{i_2}\dotsb \binom{2i_t}{i_t}= 4^n
  \binom{n+\frac{t}{2}-1}{n}.
\end{equation}
where $i_1,\dotsc,i_n$ are any nonnegative integers.

Denote by $S_t (n)$ the left hand-side of \eqref{CXa}. We remark that
$S_1 (n) = \binom{2n}{n}$ and, by \eqref{mainId}, $S_2(n) = 4^n$.
Note also that \eqref{CXa} can be obtained using induction and the
following lemma. Finally, we observe that, when $t=2\nell+1$, $4^n
\binom{n+\frac{2 \nell +
    1}{2}-1}{n}=\frac{\binom{2n+2\nell}{2n}}{\binom{n+\nell}{n}}
\binom{2n}{n}=
\frac{\binom{2n+2\nell}{n+\nell}}{\binom{2\nell}{\nell}}
\binom{n+\nell}{n}$.

\begin{lemma} For every positive integer $t$ and every nonnegative
  integer $n$,
  \begin{equation*}
    S_{t+2} (n+1) = S_t (n+1) + 4\,S_{t+2} (n)
  \end{equation*}
\end{lemma}
\begin{proof}
  In fact, $S_{t+2} (n+1) = \sum_{j=0}^{n+1} S_2 (n+1-j) S_t (j) = S_t
  (n+1) + 4 \sum_{j=0}^{n} 4^{n-j} S_t (j).$
\end{proof}

We now prove identity~\eqref{eq2} by using the inclusion-exclusion
principle and then we generalize it in Lemma~\ref{newlemma},
below. Note that the result is valid for every $\ell\in\R$, since
$\sum_{i+j=n} \binom{2i-\ell}{i} \binom{2j+\ell}{j}$ is a polynomial
in $\ell$ (of degree at most $n$). In particular, this gives us yet
another proof of \eqref{mainId}, in fact a very short one.

\begin{lemma}\label{newnewlemma}
  For every nonnegative integer numbers $i$, $j$, $n$ and $\ell$ such
  that $\ell>2n$,
$$\sum_{i+j=n}\binom{2i-\ell}{i}\binom{2j+\ell}{j}\;=\;\displaystyle4^n.$$
\end{lemma}
\begin{proof}
First note that:
\begin{eqnarray*}
\sum_{i+j=n} \binom{2i-\ell}{i} \binom{2j+\ell}{j} & = & \sum_{i+j=n}
(-1)^i \binom{\ell-1-i}{i} \binom{2n+\ell-2i}{j} \\
& = & \sum_{i+j=n} \left[ (-1)^i \binom{\ell-1-i}{i} \sum_{k+m=j}
\binom{2n+1}{k} \binom{\ell-1-2i}{m} \right] \\
& = & \sum_{k=0}^n \left[ \binom{2n+1}{k} \sum_{i+m=n-k} (-1)^i
\binom{\ell-1-i}{i} \binom{\ell-1-2i}{m} \right]
\end{eqnarray*}
Now, since $\sum_{i+m=p} (-1)^i \binom{\ell-i}{i}\binom{\ell-2i}{m} =
\sum_{i=0}^p(-1)^i \binom{\ell-i}{\ell-p}\binom{\ell-p}{i}$, we prove
that the value of the latter sum is $1$. For this purpose, define
$\mathcal{A}$ as the collection of all subsets of $[\ell]$ with
$\ell-p$ elements and let $\mathcal{A}_x$ be the collection of those
that do not contain $x$, for $x=1,\dotsc,\ell-p$, so that
$\mathcal{A}= \{[\ell-p]\}\cup\mathcal{A}_\varnothing$, where
$\mathcal{A}_\varnothing=\bigcup_{x=1}^{\ell-p}\mathcal{A}_{x}$. Now,
let, for $\varnothing\neq T\subset [\ell-p]$,
$\mathcal{A}_T=\bigcap_{x\in T}\mathcal{A}_x$.  Since, for $i=0,\ldots,p$, there
are $\binom{\ell-p}{i}$ sets of form $\mathcal{A}_T$ with $|T|=i$, and
each one has $\binom{\ell-i}{p-i}$ elements, the result follows from
the inclusion-exclusion principle.
\end{proof}

\begin{lemma}\label{newlemma}
For every nonnegative integers $i$, $j$ and $n$ and
real numbers $a,\ell$, we have:
$$\sum_{i+j=n}\binom{2i+a}{i}\binom{2j}{j}=
\sum_{i+j=n}\binom{2i+a-\ell}{i}\binom{2j+\ell}{j}.$$
\end{lemma}
\begin{proof}
  We want to prove that, for fixed $n\in\N_0$ and $a\in\R$,
  $p=\sum_{i+j=n}\binom{2i+a-\ell}{i}\binom{2j+\ell}{j}$ has degree zero as
  a polynomial (in $\ell$). Note that
\begin{align*}
  &p\;=\;\frac{1}{n!}\sum_{i=0}^n (-1)^i \binom{n}{i} q_i
  \;=\;\frac{(-1)^n}{n!}\Delta^n\,q_0, \intertext{where, for
    $i=0,1,\dotsc,n$, $p_i=\big(\ell-a+i-1\big)_i\
    \big(\ell+2n\big)_{n-i}$ and $q_i$ is obtained from $p_i$ by
    replacing $\ell$ with $\ell-2i$, so that both are polynomials of
    degree $n$ in $\ell$, and where we write, as usual,
    $\Delta\,p_i=p_{i+1}-p_i$. Now, induction on $m$ immediately shows
    that, for $1\leq m\leq n-i$,} &\Delta^m p_i=(-1)^m
  \big(a+n+m\big)_m\
  \big(\ell-a+i-1\big)_i\ \big(\ell+2n\big)_{n-i-m}\\
  \intertext{which is either zero or has degree $n-m$. We may rewrite
    each $p_i$ as } &p_i=r_n+r_{n-1}\,\ell+\dotsb+r_0\,\ell^n,
  \intertext{for suitable polynomials $r_\nell$ in $i$
    ($\nell=0,\dotsc,n$) with $r_0=1$, so that}
  &\Delta^mp_i=\Delta^mr_n+\Delta^mr_{n-1}\,\ell+\dotsb+\Delta^mr_0\,\ell^n.\\
  \intertext{Hence,
    $\Delta^m\,r_{m-1}=\dotsb=\Delta^m\,r_1=\Delta^m\,r_0=0$, which
    proves that $r_\nell$ has degree at most $\nell$ and thus}
  &q_i=r_n+r_{n-1}(\ell-2i)+\dotsb+(\ell-2i)^n,\\
  &\hphantom{q_i}=s_n+s_{n-1}\,\ell+\dotsb+\ell^n,
\end{align*}
where each $s_\nell$ is a polynomial of degree less than or equal to
$\nell$ in $i$. Therefore, $p$ is constant in $\ell$.
\end{proof}

As a clear consequence, we obtain for every nonnegative integers $t$,
$n$, and $i_1,i_2,\dotsc,i_t$ and for every real numbers
$\ell_1,\ell_2,\dotsc,\ell_t$ such that
$\ell=\ell_1+\ell_2+\dotsb+\ell_t$,
$$
\sum_{i_1 + \cdots + i_t = n} \binom{2i_1 + \ell}{i_1}\binom{2i_2}{i_2}
\dotsb \binom{2i_t}{i_t} =\sum_{i_1 + \cdots + i_t = n} \binom{2i_1 +
  \ell_1}{i_1}\binom{2i_2 + \ell_2}{i_2} \dotsb \binom{2i_t +\ell_t}{i_t}
$$
Whence we obtain the following generalization of both \eqref{eq2} and
\eqref{CXa}.
\begin{corollary}\label{newgen}
  Let $\ell_1, \ldots, \ell_t$ be any real numbers such that $\ell_1 +
  \cdots + \ell_t = 0$. Then
$$
\sum_{i_1 + \cdots + i_t = n} \binom{2i_1 + \ell_1}{i_1}\binom{2i_2 +
  \ell_2}{i_2} \dotsb \binom{2i_t + \ell_t}{i_t} =
\frac{4^n}{n!}\frac{\Gamma(n+\frac{t}{2})}{\Gamma(\frac{t}{2})}
$$
\end{corollary}

\section{Generating functions}
\noindent
In what follows, we denote by $f^{(n)}$ the $n$-th derivative of a
function $f$ of one real variable, $g(x) = \sum_{n \geq 0}
\binom{2n}{n} \, x^n$ is the generating function of the central
binomial coefficients and $C(x) = \sum_{n \geq 0} \frac{1}{n+1}
\binom{2n}{n} \, x^n$ is the generating function of the Catalan
numbers.  We remember that $g(x)=\frac{1}{\sqrt{1-4x}}$ and $C(x) =
\frac{2}{1 + \sqrt{1-4x}}$. Note that $g' = 2g^3$ and $C' = g \, C^2$.
In this section we obtain the sequences with generating functions
$g^t$, $g \, C^\ell$ and $C^\ell$, for every $t, \ell \in \mathds{R}$.

\begin{lemma} \label{deriv} \hfill
\begin{enumerate}
\item \label{gtn} For every real number $t$ and nonnegative integer $n$,
\begin{equation*}
  \frac{\left( g^t \right)^{(n)}}{n!} = 4^n \binom{n + \frac{t}{2} -1}{n} g^{t+2n}.
\end{equation*}
\item \label{gCkn} For every real number $\ell$ and nonnegative integer $n$,
\begin{equation*}
  \frac{\left( g \, C^\ell \right)^{(n)}}{n!} = \sum_{i=0}^n \binom{2n-i}{n-i} 
  \binom{\ell+i-1}{i} g^{1+2n-i} C^{\ell+i}.
\end{equation*}
\item \label{Ckn} For every real number $\ell$ and positive integer $n$,
\begin{equation*}
\left( C^\ell \right)^{(n)} = \left( \ell \, g \, C^{\ell+1} \right)^{(n-1)}.
\end{equation*}
\end{enumerate}
\end{lemma}

\begin{proof}
  Both (\ref{deriv}.\ref{gtn}) and (\ref{deriv}.\ref{gCkn}) are easily
  shown using induction and (\ref{deriv}.\ref{Ckn}) is obvious.
\end{proof}

Our next theorem is the main result of this section. We remark that
(\ref{GenFun}.\ref{Ck}) was proved by Eug\`ene Catalan in 1876
(V. \cite[p. 62 and Errata]{cat}).

\begin{theorem}[Generating functions for $g^t$, $g \, C^\ell$ and
  $C^\ell$]\label{GenFun} \hfill
\begin{enumerate}
\item \label{gt} For every real number $t$,
\begin{equation*}
g(x)^t = \sum_{n \geq 0} 4^n \binom{n + \frac{t}{2} -1}{n} x^n.
\end{equation*}
\item \label{gCk} For every real number $\ell$,
\begin{equation*}
g(x)\, C(x)^\ell = \sum_{n \geq 0} \binom{2n+\ell}{n} x^n.
\end{equation*}
\item \label{Ck} For every real number $\ell$,
\begin{equation*}
  C(x)^\ell = 1 + \sum_{n \geq 1} \frac{\ell}{2n+\ell} \binom{2n+\ell}{n} x^n. 
\end{equation*}
\end{enumerate}
\end{theorem}

\begin{proof}
  The identity (\ref{GenFun}.\ref{gt}) follows immediately from
  (\ref{deriv}.\ref{gtn}) whereas (\ref{GenFun}.\ref{Ck}) follows from
  (\ref{deriv}.\ref{Ckn}) and (\ref{GenFun}.\ref{gCk}).  If we show
  that
\begin{equation*}
\sum_{i=0}^n \binom{2n-i}{n-i} \binom{\ell+i-1}{i} = \binom{2n+\ell}{n} \ ,
\end{equation*}
for all nonnegative integer $n$ and real $\ell$, the result follows
from Lemma~\ref{deriv}. Let
\begin{equation*}
F(n,i) = \binom{2n-i}{n-i} \binom{\ell+i-1}{i}.
\end{equation*}
By Zeilberger's algorithm\cite{npwz,aeqb}, as implemented in
\emph{Mathematica} by Peter Paule and Markus Schorn \cite{pp} and by
Christian Krattenthaler \cite{krat}, we know that $T(n) = \sum_{i=0}^n
F(n,i)$ verifies
\begin{equation}\label{zeil}
(2n+\ell+1)(2n+\ell+2) \, T(n) - (n+\ell+1)(n+1) \, T(n+1) = 0,
\end{equation}
which is also verified by $T(n) = \binom{2n+\ell}{n}$, and for both it
holds $T(0)=1$. In fact, we can see that for $0 \leq i \leq n+1$,
\begin{equation*}
  (2n+\ell+1)(2n+\ell+2) \, F(n,i) - (n+\ell+1)(n+1) \, F(n+1,i) = G(n,i+1)-G(n,i)
\end{equation*}
with $G(n,i) = i(i+1) \binom{2n+1-i}{n+1-i} \binom{\ell+i}{i+1}$.
Hence, \eqref{zeil} holds since $F(n,n+1)=G(n,n+2)=G(n,0)=0$.
\end{proof}

\section{Acknowledgments}

We thank warmly Christian Krattenthaler and Cyril Banderier for a very
helpful conversation (that, in particular, introduced us to the
beautiful world of $A=B$ \cite{aeqb}) about the results of this
article. We also thank the \emph{The On-Line Encyclopedia of Integer
  Sequences} \cite{oeis} that was very useful throughout our work. The
work of both authors was supported in part by the European Regional
Development Fund through the program COMPETE --Operational Programme
Factors of Competitiveness (``Programa Operacional Factores de
Competitividade'') - and by the Portuguese Government through FCT --
Funda\c{c}\~ao para a Ci\^encia e a Tecnologia, under the projects
PEst-C/MAT/UI0144/2011 and PEst-C/MAT/UI4106/2011.

\end{document}